\documentclass[a4paper]{article}

\usepackage{amsmath}
\usepackage{amsthm}
\usepackage{amsfonts}
\usepackage{amssymb}

\usepackage[T1]{fontenc}
\usepackage[utf8]{inputenc}
\usepackage[polish, english]{babel}

\usepackage{authblk}

\usepackage{bbm}
\usepackage{graphicx}
\usepackage{hyperref}
\usepackage{enumitem}

\newtheorem{theorem}{Theorem}
\newtheorem{lemma}{Lemma}

\newcommand\widebar[1]{\mathop{\overline{#1}}}
\newcommand\BIN{\textrm{BIN}}
\newcommand\E{\mathbbm{E}}

\title{Conflict-free chromatic number vs conflict-free chromatic index}

\author[1,2]{Michał Dębski}
\author[3]{Jakub Przybyło}
\affil[1]{Faculty of Informatics, Masaryk University Botanick\'a 68a, 602 00 Brno, Czech Republic}
\affil[2]{Faculty of Mathematics and Information Sciences, Warsaw University of Technology, Koszykowa 75, 00-662 Warszawa, Poland}
\affil[3]{AGH University of Science and Technology, al. A. Mickiewicza 30, 30-059 Krakow, Poland}

\begin{document}

\maketitle

\begin{abstract}
A vertex coloring of a given graph $G$ is conflict-free if the closed neighborhood of every vertex contains a unique color (i.e. a color appearing only once in the neighborhood). The minimum number of colors in such a coloring is the conflict-free chromatic number of $G$, denoted $\chi_{CF}(G)$. What is the maximum possible conflict-free chromatic number of a graph with a given maximum degree $\Delta$? Trivially, $\chi_{CF}(G)\leq \chi(G)\leq \Delta+1$, but it is far from optimal -- due to results of Glebov, Szabó and Tardos, and of Bhyravarapu, Kalyanasundaram and Mathew, the answer in known to be $\Theta\left(\ln^2\Delta\right)$.

We show that the answer to the same question in the class of line graphs is  $\Theta\left(\ln\Delta\right)$ -- that is, the extremal value of the conflict-free chromatic index among graphs with maximum degree $\Delta$ is much smaller than the one for conflict-free chromatic number.
The same result for $\chi_{CF}(G)$ is also provided
in the class of near regular graphs, i.e. graphs with minimum degree $\delta \geq \alpha \Delta$.
\end{abstract}

\section{Introduction}

We say that a coloring of vertices of a graph $G$ is \emph{conflict-free} if for every vertex $v\in V(G)$ there is a color that appears exactly once in the closed neighborhood of $v$. The minimum number of colors in such a coloring is called the \emph{conflict-free chromatic index} of $G$ and denoted by $\chi_{CF}(G)$. In the literature, this invariant is also called the closed neighborhood conflict-free chromatic number of $G$ or the conflict-free chromatic parameter of $G$ and sometimes denoted by $\kappa_{CF}(G)$.

Note that every proper vertex coloring of $G$ is conflict-free (because the color of a vertex  $v$ cannot be used on any neighbor of $v$), so $\chi_{CF}(G)\leq \chi(G)$. However, we can usually find a conflict-free coloring that uses much fewer colors. For example, odd cycles require only $2$ colors instead of $3$, and cliques require only $2$ colors instead of $n$.

This model of coloring is originally motivated by channel assignment in wireless networks. Here colors correspond to available frequencies and each node $v$ in the network wants to receive a transmission on some frequency $f_v$ -- it is possible only if $v$ is within range of exactly one transmitter that uses frequency $f_v$, as otherwise the signals would interfere. In a general setting this yields a hypergraph coloring problem that can be specialized to conflict-free coloring of graphs -- see \cite{EvenEtAl,SmorodinskyApplications,SmorodinskyPhd}.

Our work stems from the question: what is the maximum possible conflict-free chromatic number of a graph with a given maximum degree $\Delta$? In 2009 Pach and Tardos showed that the answer is of order at most $\ln^{2+\epsilon}\Delta$ and at least $\ln \Delta$ \cite{PachTardos}. Both bounds have been improved to $\Theta\left(\ln^2\Delta\right)$; Glebov, Szabó and Tardos showed that certain random graphs on $n$ vertices require at least $\Omega\left(\ln^2n\right)$ colors \cite{GlebovEtAl}, while Bhyravarapu, Kalyanasundaram and Mathew gave a randomized procedure that constructs conflict-free colorings using $O\left(\ln^2\Delta\right)$ colors \cite{Hindusi}. Constants hidden in the $\Omega,O$-notations are different, as one might expect from probabilistic proofs, but if we focus only on the order of magnitude -- the question is completely answered and the mentioned results can be thought of as a conflict-free analog of Brooks' theorem.

Random graphs used by Glebov, Szabó and Tardos have vertices of different (expected) degrees, ranging from $n^{-\alpha}$, for some constant $\alpha$ close to $1$, to $\frac{n}{\ln n}$. It turns out that such wide spread of degrees is essential; a theorem by Kostochka, Kumbhat and Łuczak implies that if a graph is $\Delta$-regular, then it admits a conflict-free coloring using only $O(\ln \Delta)$ colors \cite[Theorem 6]{KostochkaEtAl}. We prove a slightly stronger statement that the logarithmic upper bound holds if all degrees are of the same order of magnitude.

\begin{theorem}
\label{thm_regularGraphs}
Let $\alpha>0$. If $G$ is a graph with maximum degree $\Delta$ and $\delta(G)\geq \alpha\Delta$, then the conflict-free chromatic number of $G$ is at most $O\left(\ln \Delta\right)$.
\end{theorem}

Can we prove the same, logarithmic upper bound for graphs with arbitrary distribution of vertex degrees? Of course it is impossible in general, but there is hope for such a result in a restricted class of graphs. As our main theorem we prove a logarithmic upper bound that applies to line graphs.

\begin{theorem}
\label{thm_lineGraphs}
The conflict-free chromatic number of any line graph $G$ of maximum degree $\Delta$ is at most
$$
O\left(\ln\Delta\right).
$$
\end{theorem}

We also show that Theorem \ref{thm_lineGraphs} is tight up to a constant multiplicative factor, as witnessed by line graphs of complete graphs. Note that this result gives a logarithmic (in $\Delta$) lower bound on the conflict-free chromatic number for a specific family of graphs, unlike other lower bounds mentioned, based on probabilistic constructions.

\begin{theorem}
\label{thm_cliqueLineGraph}
Let $G$ be a line graph of $K_n$ for some $n>1$. The conflict-free chromatic number of $G$ is at least $\Omega\left(\ln n\right)$.
\end{theorem}

There is another way of formulating Theorem \ref{thm_lineGraphs}. We can say that an edge coloring of a graph $G$ is \emph{conflict-free} if for every edge $e\in E(G)$ there is a color that appears exactly once among edges that share a vertex with $e$ (including $e$ itself), and refer to the minimum number of colors in such a coloring as the \emph{conflict-free chromatic index} of $G$. In this language, Theorem \ref{thm_lineGraphs} says that the conflict-free chromatic index of a graph with maximum degree $\Delta$ is of order at most $\ln\Delta$. 
In view of Theorem~\ref{thm_cliqueLineGraph}, it therefore may be regarded as order-wise analog of Vizing's theorem,
in the same way that aforementioned results were related to Brooks' theorem.
It is thus a somewhat surprising phenomenon that unlike in the case of the fundamental results of Vizing and Brooks,
the maximal in terms of $\Delta$ values of both conflict-free invariants are of distinct magnitudes, and that it is the edge coloring variant that requires less colors (note that a proper edge coloring, contrary to a proper vertex coloring, does not have to be conflict free).

The rest of the paper is organized as follows. In Section \ref{section_prelims} we discuss probabilistic tools that will be used. In Section \ref{section_SideProofs} we give proofs of Theorems \ref{thm_regularGraphs} and \ref{thm_cliqueLineGraph}. Section \ref{section_mainProof} is devoted to the proof of the main result, Theorem \ref{thm_lineGraphs}. We conclude with Section \ref{section_final}, where we discuss some possible directions of 
further research.

\section{Preliminaries}
\label{section_prelims}

We will use three fairly standard probabilistic tools. The first one is the Lov\'asz Local Lemma \cite{LLLRef}.

\begin{theorem}[The Local Lemma]
\label{thm_lll}
Let $A_1,A_2,\ldots,A_n$ be events in an arbitrary probability space. Suppose that each event $A_i$ is mutually independent of a set of all the other events $A_j$ but at most $D$, and that $\Pr(A_i)\leq p$ for all $1\leq i \leq n$. If
$$
ep(D+1)\leq 1
$$
then $\Pr\left(\bigcap_{i=1}^n \widebar{A_i}\right)>0$.
\end{theorem}

We will also need the Chernoff bound in a standard version \cite{MolloyReedColoringBook}.

\begin{theorem}[Chernoff Bound]
\label{thm_chernoff}
For any $0\leq t \leq np$,
$$
\Pr(\BIN(n,p)>np+t)\leq e^{-\frac{t^2}{3np}} \textrm{ and } \Pr(\BIN(n,p)<np-t)\leq e^{-\frac{t^2}{2np}},
$$
where $\BIN(n,p)$ is the sum of $n$ independent Bernoulli variables, each equal to $1$ with probability $p$ and $0$ otherwise.
\end{theorem}

Our final tool is Talagrand's Inequality, in a slightly weaker, but more convenient version from a paper by Molloy and Reed \cite{MolloyTalagrandRef}.

\begin{theorem}[Talagrand's Inequality]
\label{thm_talagrand}
Let $X$ be a nonnegative random variable determined by $\ell$ independent trials $T_1,\ldots,T_n$. Suppose there exist constants $c,k>0$ such that for every set of possible outcomes of the trials, we have:
\begin{enumerate}
\item changing the outcome of any one trial can affect $X$ by at most $c$, and
\item for each $s>0$, if $X\geq s$ then there is a set of at most $ks$ trials whose outcomes certify that $X\geq s$.
\end{enumerate}
Then for any $t\geq 0$ we have
$$
\Pr\left(\left|X-\E(X)\right| > t+20c\sqrt{k\E(X)}+64c^2k\right)\leq 4e^{-\frac{t^2}{8c^2k(\E(X)+t)}}.
$$
\end{theorem}

\section{Proofs of Theorems \ref{thm_regularGraphs} and \ref{thm_cliqueLineGraph}}
\label{section_SideProofs}

We begin with the proof of Theorem \ref{thm_regularGraphs} (a logarithmic bound for near regular graphs). In our argument we start by randomly selecting a subset $V'\subseteq V(G)$ of vertices such that every vertex of $G$ has $\Theta\left(\ln \Delta\right)$ neighbors in $V'$ and then color $V'$ randomly, using $\Theta\left(\ln \Delta\right)$ colors. We show that every vertex of $G$ sees some color exactly once with probability at least $1-e^{-\Theta(\ln \Delta)}$ which, after a right choice of constants, is just enough to use the Lov\'asz Local Lemma to guarantee that every vertex sees some color exactly once. 

\begin{proof}[Proof of Theorem \ref{thm_regularGraphs}]
Wherever needed we assume that the maximum degree $\Delta$ of $G$ is large enough.
We start by picking a set $V'\subseteq V(G)$ such that for every vertex $v\in V(G)$ we have 
$$350\ln\Delta\leq\left|V' \cap N[v]\right|\leq \frac{450}{\alpha}\ln\Delta.$$

Let $V'$ be selected randomly where each vertex of $G$ is placed in $V'$ with probability $\frac{400 \ln\Delta}{\alpha\Delta}$, independently at random. We will show that the desired property holds with positive probability.

For a vertex $v\in V(G)$, let $A_v$ be the event that $\left|V' \cap N[v]\right|<350\ln\Delta$. Note that $\left|V' \cap N[v]\right|$ is greater or equal to a random variable with distribution $\BIN(\alpha \Delta,\frac{400 \ln\Delta}{\alpha\Delta})$. Therefore, by using Theorem \ref{thm_chernoff} with $t=50\ln\Delta$, we get that
$$
\Pr(A_v)\leq \Pr\left(\BIN(\alpha \Delta,\frac{400 \ln\Delta}{\alpha\Delta}) < 350\ln\Delta\right) \leq e^{-\frac{50^2\ln^2\Delta}{800\ln\Delta}}\leq \Delta^{-3}.
$$

Now let $B_v$ be the event that $\left|V' \cap N[v]\right|> \frac{450}{\alpha}\ln\Delta$. Similarly as above, $\left|V' \cap N[v]\right|$ is less or equal to a random variable with distribution $\BIN(\Delta+1,\frac{400 \ln\Delta}{\alpha\Delta})$, so by Theorem \ref{thm_chernoff} with $t=\frac{50}{\alpha}(1-\frac{8}{\Delta})\ln\Delta$ we obtain that
$$
\Pr(B_v)\leq \Pr\left(\BIN(\Delta+1,\frac{400 \ln\Delta}{\alpha\Delta}) > \frac{450\ln\Delta}{\alpha}\right) \leq e^{-\frac{50^2(1-\frac{8}{\Delta})^2\ln^2\Delta}{1200\alpha(1+\frac{1}{\Delta})\ln\Delta}}\leq \Delta^{-2.05}.
$$

Note that each of the events $A_v$ and $B_v$ is mutually independent of the set of all other events $A_u$ and $B_u$ for all $u$ which are at distance at least $3$ from $v$ in $G$. Therefore, by Theorem \ref{thm_lll} with $p=\Delta^{-2.05}$ and $D=2\Delta^2+1$ we conclude that $\Pr\left(\bigcap_{v\in V(G)}\widebar{A_v}\cap \widebar{B_v}\right)>0$, hence the desired set $V'$ exists.

Now we construct a conflict-free coloring of $G$ such that all vertices outside $V'$ are colored with the same color, and vertices from $V'$ are colored using further $\lceil\frac{2700}{\alpha}\ln\Delta\rceil$ colors, independently and uniformly at random. 

For $v\in V(G)$, let $X_v$ be a random variable that counts the number of vertices $w$ in $N[v]\cap V'$ such that the color of $w$ is the same as the color of some other vertex from $N[v]\cap V'$. Let $d_v:=\left|N[v]\cap V'\right|$. Note that the coloring is conflict-free if $X_v<d_v$ for every $v\in V(G)$; now we will show that this happens with positive probability.

Given $v$ and $w\in N[v]\cap V'$, the probability that $w$ is counted in $X_v$ (i.e. that the color of $w$ appears in $N[v]\cap V'\setminus \lbrace w\rbrace$) is at most
$$
1-\left(1-\frac{1}{\frac{2700}{\alpha}\ln\Delta}\right)^{d_v-1}
\leq 1-\left(1-\frac{1}{\frac{2700}{\alpha}\ln\Delta}\right)^{\frac{450}{\alpha}\ln\Delta-1} \leq 1-e^{-\frac{1}{6}} \leq \frac{1}{6}.
$$
Therefore, $\E(X_v)\leq \frac{1}{6}d_v$.

Note that $X_v$ satisfies the assumptions of Theorem \ref{thm_talagrand} with $\ell=d_v$ and $c=k=2$, where the trial $T_i$ corresponds to the choice of color for the $i$-th vertex from $N[v]\cap V'$. Indeed, recoloring a single vertex can change $X_v$ by at most $2$ and $X_v\geq s$ can be certified by colors of $s$ vertices from $N[v]\cap V'$ and, for each of them, a color of some other vertex from $N[v]\cap V'$. Therefore, by Theorem \ref{thm_talagrand} with $t=\frac{1}{2}d_v$, we obtain that
$$
\Pr\left(\left|X_v-\E(X_v)\right|>\frac{1}{2}d_v+40\sqrt{2\E(X_v)}+512\right)\leq 4e^{-\frac{\frac{1}{4}d_v^2}{64(\E(X_v)+\frac{1}{2}d_v)}}.
$$
Since $\E(X_v)\leq \frac{1}{6}d_v$, for $d_v$ large enough (i.e. for $\Delta$ sufficiently large), we thus obtain that
$$
\Pr(X_v = d_v)\leq 4e^{-\frac{\frac{1}{4}d_v^2}{\frac{128}{3}d_v}} = 4e^{-\frac{3}{512}d_v} \leq 4e^{-\frac{1050}{512}\ln\Delta}\leq 4\Delta^{-2.05}.
$$

Note that the event $X_v=d_v$ is mutually independent of the set of all events $X_u=d_u$ for $u$ that are at distance at least $3$ from $v$ in $G$. Therefore, by Theorem \ref{thm_lll}, with $p=4\Delta^{-2.05}$ and $D=\Delta^2$, we conclude that with positive probability $X_v<d_v$ for all $v\in V(G)$. Hence, the obtained coloring is conflict-free, which completes the proof.
\end{proof}

Now we proceed to the proof of Theorem \ref{thm_cliqueLineGraph}. We will say that for a graph $H$, an edge-coloring $f$ of $H$ and an edge $e\in E(H)$, the edge $e$ is \emph{satisfied by $f$ with the color $c$} if exactly one of the edges adjacent to $e$ (including $e$) is colored with the color $c$; similarly, an edge is \emph{satisfied} by $f$ if it is satisfied by $f$ with at least one color.

\begin{proof}[Proof of Theorem \ref{thm_cliqueLineGraph}]
%
Fix $n>1$ and set $x:=\left\lfloor \log_2 n - \log_2\log_2 n - 1 \right\rfloor$. Let $f$ be a coloring of edges of $K_n$ that uses $x$ colors. We will show that at least one edge of $K_n$ is not satisfied by $f$, which immediately completes the proof.

Consider a relation $\sim$ defined on vertices of $K_n$ such that $u\sim v$ iff the set of colors used by $f$ on edges incident to $u$ is equal to the set of colors used by $f$ on edges incident to $v$. Note that $\sim$ is an equivalence relation.

Let $A$ be the largest equivalence class of $\sim$. Since there are at most $2^{x}$ of them, it follows that $\left|A\right|\geq \frac{n}{2^x}$, and by the choice of $x$ we conclude that $\left|A\right|> x+1$.

Now consider any two vertices $u,v\in A$. Note that, since $u\sim v$, the edge $uv$ can be satisfied by $f$ with a color $c$ only if $f(uv)=c$ and $uv$ is the only edge colored with $c$ that is incident to $u$ or $v$. Therefore, the set of all edges with both endpoints in $A$ satisfied by any given color $c$ is a matching. Recall that there are $x$ colors, so all edges with both endpoints in $A$ satisfied by $f$ can be covered by $x$ matchings. Since $\left|A\right|> x+1$, some edges in $A$ are thus not satisfied, which completes the proof.
\end{proof}

\section{Proof of the Main Theorem}
\label{section_mainProof}

In this section we prove Theorem \ref{thm_lineGraphs}, which implies the $O\left(\ln\Delta\right)$ bound for line graphs. In what follows $H$ is used to denote an arbitrary graph and $G$ will be the line graph of $H$ -- that is, we will be coloring edges of $H$ and vertices of $G$.

The proof is iterative. At each step we randomly pick a very small set $S$ of edges of $H$ -- where very small means that every vertex of $H$ is expected to be incident with a constant number of edges from $S$. Then we find a coloring $f$ of $S$ such that for every high-degree vertex $v$ of $H$ at least a constant fraction of edges incident to $v$ is satisfied by $f$ -- recall that an edge $e$ of a graph $H$ is \emph{satisfied} by a partial edge coloring of $H$ if there is a color $c$ such that $e$ is adjacent to exactly one edge of color $c$. Then, we forget about all satisfied edges and proceed to the next step, where we use a new set of colors. Note that after a logarithmic (in $\Delta$) number of such steps we will be left with a graph that has a constant maximum degree, and so its edges can be colored with a constant number of colors -- that is how $\ln \Delta$ appears in the proof.

In order to succeed we need to ensure that at each step a constant number of colors is used, which requires a proper selection of $S$. For each vertex $v$ of sufficiently high degree we will add to $S$ one edge incident to $v$ -- this way, each connected component of $H[S]$ will be either a tree or a unicyclic graph. We start by showing that such graphs admit $3$-colorings of edges that are conflict-free and have a unique color around each vertex; those colorings will be used to construct $f$.

\begin{lemma}
\label{lemma_uniqueEdge}
Let $H$ be a graph such that each connected component of $H$ contains at most one cycle. There is a conflict-free coloring of edges of $H$ with $3$ colors such that for every vertex $v$ there is a color $c_v$ such that $v$ is incident with exactly one edge of color $c_v$.
\end{lemma}
\begin{proof}
We will show an explicit construction of the desired coloring using $\lbrace 0,1,2\rbrace$ as the set of colors. We will use addition modulo $3$ -- that is, for a color $a\in \lbrace 0,1,2\rbrace$ we will refer to two other colors as $a+1$ and $a+2$.
Note that it is enough to prove the lemma when $H$ is connected. Therefore, we will consider two cases.

First suppose that $H$ contains exactly one cycle. Start by taking a proper edge-coloring $c$ of the cycle that uses $3$ colors. For each vertex $v$ of the cycle, let $m(v)$ be the color that is not assigned in $c$ to any edge incident to $v$. Next extend the function $m$ to all vertices of $H$ so that $m(u)=m(v)+1$ whenever $u$ is a child\footnote{We say that $u$ is a child of $v$ if the two vertices are adjacent and the shortest path from $u$ to the cycle goes through $v$.} of $v$. Now we extend $c$ to all edges of $H$ so that for each vertex $v$, all edges from $v$ to its children are colored by $m(v)$.

Note that $c$ is a conflict-free coloring of $H$, 
because each edge on the cycle is colored differently than its adjacent edges, 
while every edge $uv$ where $u$ is a child of $v$ shares a vertex with exactly one edge of color $m(v)-1$. Moreover, each vertex $v$ on the cycle is incident to exactly one edge of color $m(v)+1$ and $m(v)+2$, and each vertex $u$ outside the cycle is incident to exactly one edge of color $m(u)-1$. Hence, $c$ satisfies the desired properties and the proof is complete in this case.

Now suppose that $H$ contains no cycles, i.e. it is a tree. Pick some vertex $r$ of $H$ of degree $1$ as the root and set $m(r)=0$. Then proceed as in the previous case, i.e. extend $m$ to all vertices of $H$ so that $m(u)=m(v)+1$ for each child $u$ of $v$ and color all edges between a vertex $v$ and its children 
with the color $m(v)$. Note $r$ sees exactly one edge, of color $m(r)$, and this edge does not share a vertex with any other edge of color $m(r)$. For other vertices and edges we can use the same argument as in the first case to conclude that the resulting coloring satisfies the desired properties, hence the proof is complete.
\end{proof}

Now we show the main lemma that corresponds to a single step in the iterative procedure outlined above. For a set of edges $L$ of a graph $H$, by $H\setminus L$ we mean the graph $(V(H),E(H)\setminus L)$.

\begin{lemma}
\label{lemma_lineGraphs}
Let $H$ be a graph with sufficiently large maximum degree $\Delta$. There exists a set $S\subseteq E(H)$ and a coloring $f$ of $S$ using $9$ colors such that 
$$
\Delta\left(H \setminus L\right)\leq \left(1-e^{-4}\right)\Delta \textrm{ ~~and~~ } S\subseteq L
$$
where $L$ is the set of edges satisfied by $f$.
\end{lemma}
\begin{proof}
We choose $S$ and $f$ by the following random procedure. First, for each vertex $v$ of degree at least $\frac{\Delta}{2}$ we pick an edge $e_v$ incident to $v$ and a color $f_1(v)$ that is either $0$ or $1$; all the choices are made independently and uniformly at random. Then we take $S:=\bigcup\lbrace e_v\rbrace$ and define $f_2:S\rightarrow\lbrace 0, 1, 2\rbrace$ such that for every $v$ we assign the color $f_1(v)$ to the edge $e_v$, and then assign the color $2$ to all edges that were assigned both: $0$ and $1$ 
(note that this may happen for an edge $e=uv$ only if $e=e_u=e_v$).

We will show that with positive probability the following property holds:
\begin{enumerate}
\item[(a)]\label{con_2nd} for every vertex $v$ of degree greater than $\frac{\Delta}{2}$, for $i=0,1$, at least $e^{-4}\Delta$ neighbors of $v$ are not adjacent to any edge of color $i$ or $2$.
\end{enumerate}

For a vertex $v$ of degree greater than $\frac{\Delta}{2}$, take $B_v$ to be the event that $v$ violates condition (a). 
For $i\in\lbrace 0,1\rbrace$ take $Y_v^i$ to be a random variable that counts the number of neighbors of $v$ that are incident to at least one edge from $S$ that was assigned color $i$ or $2$. 

For any vertex $u$, changing $f_1(u)$ or $e_u$ can change the value of $Y_v^i$ by at most $2$, and $Y_v^i\geq s$ can be witnessed by the values of $f_1(u)$ and $e_u$ for at most $s$ vertices $u$. Therefore, we can use Talagrand's Inequality (Theorem \ref{thm_talagrand}) for $Y_v^i$ with $k=c=2$. In order to do so, we will now estimate $\E(Y_v^i)$.

Note that $Y_v^i$ is equal to $\sum_{u\in N(v)}I^i_u$, where $I^i_u$ is a $0-1$ variable that indicates if $u$ is incident to an edge that was assigned color $i$ or $2$. Note that a sufficient condition for $I^i_u$ to be $0$ is that $f_1(u)\neq i$ and for every neighbor $w$ of $u$ we do not have both $f_1(w)=i$ and $e_w=uw$. Recall that the degree of $u$ is at most $\Delta$ and, by our random procedure, all neighbors $w$ of $u$ for which $e_w$ is defined have degree at least $\frac{\Delta}{2}$. Therefore, $I^i_u$ is $0$ with probability at least $\frac{1}{2}\left(1-\frac{1}{2}\frac{2}{\Delta}\right)^\Delta\geq e^{-2}$. It follows that $\E(I_u^i)\leq 1-e^{-2}$ and, by linearity of expected value, $\E(Y_v^i)\leq \deg(v)-e^{-2}\deg(v)\leq \deg(v)-e^{-3}\Delta$. 

Now we apply Theorem \ref{thm_talagrand} to $Y_v^i$ with $k=c=2$ and $t=e^{-4}\Delta$. Using the fact that $\E(Y_v^i)\leq \deg(v)-e^{-3}\Delta$ and more convenient estimations $\E(Y_v^i)\leq \E(Y_v^i)+t \leq \Delta$ we obtain that
$$
\Pr\left(Y_v^i > \deg(v) - e^{-3}\Delta + e^{-4}\Delta + 40\sqrt{2\Delta} + 512 \right)\leq 4e^{-\frac{e^{-8}\Delta^2}{64\Delta}}.
$$
For $\Delta$ sufficiently large, we thus have
$$
\Pr\left(Y_v^i > \deg(v) - e^{-4}\Delta\right)\leq 4e^{-\frac{e^{-8}\Delta^2}{64\Delta}}.
$$
Note that a necessary condition for $B_v$ is that either $Y_v^0>\deg(v)-e^{-4}\Delta$ or $Y_v^1>\deg(v)-e^{-4}\Delta$, so
$$
\Pr(B_v)\leq \Pr(Y_v^0>\deg(v)-e^{-4}\Delta)+\Pr(Y_v^1>\deg(v)-e^{-4}\Delta) \leq 8e^{-\frac{e^{-8}}{64}\Delta}.
$$

The event $B_v$ depends only on $f_1(u)$ and $e_u$ for all $u$ that are at distance at most $2$ from $v$, so it is mutually independent of the set of all other events $B_w$ with $w$ at distance greater than $4$ from $u$. Therefore, by Theorem \ref{thm_lll} with $p=\Delta^{-5}$ and $D=\Delta^4$ we conclude that $\Pr\left(\bigcap_{v\in V(G)}\widebar{B_v}\right)>0$, so property (a) is satisfied with positive probability.

Note that by the choice of $S$, for each subset $X\subseteq V(H)$ there are at most $\left|X\right|$ edges in $S$ with both endpoints in $X$, because each such edge must have been chosen as $e_v$ for some $v\in X$. It follows that each connected component in $H[S]$ is either a tree or contains exactly one cycle. Now take $f_3$ to be a $3$-coloring of $S$ given by Lemma \ref{lemma_uniqueEdge} applied to $H[S]$. 

We set $f$ to be a product of $f_2$ and $f_3$, i.e. $f(e)=(f_2(e),f_3(e))$, and claim that $f$ and $S$ satisfy the properties required in the lemma. 

Indeed, take $v$ to be a vertex of degree at least $\frac{\Delta}{2}$. Pick a neighbor $u$ of $v$ such that the color $c_v$ in $f_3$ of an edge $uv$ is unique among colors of edges incident with $v$. Take $i$ to be $0$ if $f_2(uv)=0$ or $1$ if $f_2(uv)\in \lbrace 1, 2 \rbrace$. Let $L_v$ be a set of all edges $vw\in E(H)$ such that $w$ is not incident to any edge colored $i$ or $2$ in $f_2$. Note that edges in $L_v$ are satisfied by $f$, because for every $vw\in L_v$, the vertex $w$ is not incident to any edge getting either color $i$ or $2$ in $f_2$. By the property (a), at least $e^{-4}\Delta$ edges incident to $v$ are in $L_v$, which proves the first part of the lemma. Since $f_3$ is a conflict-free coloring of $H[S]$, the proof of the lemma is complete.
\end{proof}

Now we are ready to prove Theorem \ref{thm_lineGraphs} -- after establishing the lemma above, all that remains is to formalize 
the way 
all the steps of the procedure are handled.

\begin{proof}[Proof of Theorem \ref{thm_lineGraphs}]
Denote by $\Delta_0$ the least integer such that Lemma \ref{lemma_lineGraphs} is true for all $\Delta\geq\Delta_0$ and $\Delta_0\geq (1-e^{-4})^{-1}$. Let $H$ be a graph with maximum degree $\Delta>\Delta_0$. We start with at most $\lfloor\log_{\frac{1}{1-e^{-4}}}\Delta\rfloor$ applications of Lemma \ref{lemma_lineGraphs}. To be more precise, we set $H_1:=H$ and then for $i=1,2,\ldots$ we define $S_i$ and $f_i$ to be the set of edges and its coloring obtained from Lemma \ref{lemma_lineGraphs} applied for the graph $H_i$, and $L_i\supseteq S_i$ to be the set of edges of $H_i$ satisfied by $f_i$; set $H_{i+1}:=H_i \setminus L_i$. Let $i_{\max}$ be the first $i$ such that $\Delta(H_i)< \Delta_0$. Let $f_{i_{\max}}$ be a proper edge-coloring of $H_{i_{\max}}$ with the least possible number of colors and set $S_{i_{\max}}:=E(H_{i_{\max}})$. Note that $L_{i_{\max}}=S_{i_{\max}}$.

Finally, we define a coloring $c$ of the edges of $H$ 
by setting  $c(e)=(f_i(e),i)$ for every $e\in E(H)$ where $i$ is the 
integer such that $e$ received a color in $f_i$;
note that such $i$ is unique, because $S_i\subseteq L_i$. Note also that for each $i$, the edges from $L_i$ are satisfied by $c$. Since each edge of $H$ belongs to $L_i$ for some $i$, $c$ is conflict-free. We have used at most $\Delta_0$ colors in step $i_{\max}$ and $9$ colors in each of the remaining steps. Hence the total number of colors that we have used is at most $9\left\lfloor\log_{\frac{1}{1-e^{-4}}}\Delta\right\rfloor+\Delta_0$. 
Since the maximum degree $\Delta'$ of the line graph $G$ of $H$ equals at least $\Delta-1$, we have thus proven that $\chi_{CF}(G) = O(\Delta')$.
\end{proof}

\section{Final remarks}
\label{section_final}


We believe that one possible direction of further research could regard generalizing our main result to 
claw-free graphs, and further to $K_{1,r}$-free graphs for $r>3$. 
Such a generalization would not be straightforward, as our proof heavily relies on a specific structure of line graphs, 
hence this will require developing some new ideas.
An especially promising subclass of $K_{1,r}$-free graphs is the class of intersection graphs of geometric objects (including, for example, unit disk graphs) -- such graphs on $n$ vertices have conflict-free chromatic number not greater than $O\left(\ln n\right)$, as proved by Keller and Smorodinsky \cite{KellerEtAl}, and improving the result to $O(\ln\Delta)$ seems plausible.


In view of Theorem~\ref{thm_regularGraphs} and the mentioned above construction of Glebov, Szabó and Tardos from~\cite{GlebovEtAl}, it would also be interesting to examine a threshold for $\delta$ (in terms of $\Delta$)
above which an $O(\ln\Delta)$ upper bound for  $\chi_{CF}(G)$ retains valid, and investigate what happens below such a threshold, as in general  $\chi_{CF}(G) \leq O(\ln^2\Delta)$.  
We believe this direction of research may provide a list of very interesting though yet unpredictable results.

\end{document}